\documentclass[1p,times]{elsarticle}
\usepackage{amsmath,amssymb}
\usepackage{latexsym}
\usepackage{amsthm}
\usepackage{graphicx}
\usepackage{url}

\newtheorem{problem}{Problem}
\newtheorem{lemma}{Lemma}[section]
\newtheorem{theorem}{Theorem}

\newtheorem{proposition}[lemma]{Proposition}

\newtheorem{fact}[lemma]{Fact}

\theoremstyle{definition}

\newtheorem{step}{Step}
\newcommand{\bst}{\begin{step}}
\newcommand{\est}{\end{step}}
\newcommand{\acf}{algebraically closed field }
\newcommand{\acfd}{algebraically closed field}

\newcommand{\fmr}{finite Morley rank }
\newcommand{\rk}{\operatorname{rk}}
\newcommand{\stab}{{\rm stab}}

\newcommand{\fmrd}{finite Morley rank}
\newcommand{\bi}{\begin{itemize}}
\newcommand{\ei}{\end{itemize}}
\newcommand{\bpr}{\begin{proof}}
\newcommand{\epr}{\end{proof}}
\newcommand{\bq}{\begin{quote}}
\newcommand{\eq}{\end{quote}}
\newcommand{\vr}[1]{\mathbf{#1}}
\def\nn{\ensuremath \mathbb{N}}

\journal{Journal of Algebra}
\begin{document}
\begin{frontmatter}

\title{Groups of finite Morley rank with a generically sharply multiply transitive action}
\date{20 July 2018}

\author{Ay\c{s}e Berkman}
\address[AB]{Mathematics Department, Mimar Sinan Fine Arts University, Istanbul, Turkey}
\ead{ayse.berkman@msgsu.edu.tr, ayseberkman@gmail.com}

\author{Alexandre Borovik}
\address[AVB]{School of Mathematics, University of Manchester, UK}
\ead{alexandre@borovik.net}

\begin{abstract}
\footnotetext{This is the final pre-publication version of the paper: A. Berkman and A. Borovik,
\emph{Groups of finite Morley rank with a generically sharply multiply transitive action}, J. Algebra (2018), \url{https://doi.org/10.1016/j.jalgebra.2018.07.033}. Accepted for publication 28 July 2018. The manuscript will undergo copyediting, typesetting, and review of the resulting proof before it is published in its final form.}

We prove that if $G$ is a group of finite Morley rank that acts definably and generically sharply $n$-transitively on a connected abelian group $V$ of Morley rank $n$ with no involutions, then there is an algebraically closed field $F$ of characteristic $\ne 2$ such that $V$ has the structure of a vector space of dimension $n$ over $F$ and $G$ acts on $V$ as the group $\operatorname{GL}_n(F)$ in its natural action on $F^n$. \end{abstract}


\end{frontmatter}

\section{Introduction}

\subsection{Groups of finite Morley rank, their actions, and binding groups}

Groups of finite Morley rank are abstract groups, possibly with additional structure,
equipped with a notion of dimension that assigns to every
definable set $X$ a natural number, called \emph{Morley rank} and
denoted by ${\rm rk}(X)$, satisfying well-known
rudimentary axioms, given for example in \cite{bn}.
Examples are furnished by algebraic groups over algebraically
closed fields, with ${\rm rk}(X)$ equal to the dimension of the
Zariski closure of $X$. Groups of finite Morley rank equipped with a definable action arise naturally as {\em binding groups} in many first order theories: for example,  Lie groups of the Picard--Vessiot theory of linear differential equations can be viewed as a special case \cite{Poisat83}. A more detailed discussion of binding groups that play in model theory a role akin to that of Galois groups could be found in the Introduction to \cite{bbpseudo}.

The present paper is one of the first steps in a research programme aimed at deeper understanding of definable actions of groups of finite Morley rank and, in particular, binding groups. To explain the programme, we need a brief overview of the current state of the classification of simple groups of finite Morley rank.

\subsection{Simple groups of finite Morley rank}

The development of the theory of groups of finite Morley rank started around 1980 in
pioneering works by Zilber~\cite{Zi-GR} and Cherlin~\cite{Ch-GSR}; they formulated what remains the central conjecture:

\bq \emph{Simple infinite groups of finite Morley rank are algebraic groups over algebraically closed fields}.
\eq

The biggest result  towards the Cherlin--Zilber Conjecture is the characterization below of algebraic groups over fields of characteristic $2$ (`even type').

\begin{fact}{\rm \cite{abc}} If a simple group $G$ of finite Morley rank
contains an infinite elementary abelian $2$-subgroup {\rm (}we say
in this situation that $G$ is of \emph{even type}{\rm )} then $G$
is isomorphic to a simple algebraic group over an algebraically
closed field of characteristic $2$. \label{th:even}
\end{fact}

In view of this result and properties of Sylow 2-subgroups in groups of finite Morley rank \cite{BBC-ID}, the remaining configurations in a proof of the Cherlin--Zilber Conjecture involve either groups of \emph{degenerate type}, that is, simple groups without involutions (here a counterexample may emerge) or groups where a Sylow 2-subgroup contains a non-trivial divisible abelian subgroup of finite index (groups of \emph{odd type}).

Since most proofs in the classification project use induction on
Morley rank, it is convenient to say that
a group of finite Morley rank is a \emph{K-group} if all its
simple definable sections (that is, groups of the form $H/K$ for
definable subgroups $K \lhd H \leqslant G$) are algebraic groups. We say
that $G$ is a \emph{$K^*$-group} if proper definable subgroups of $G$
are $K$-groups. Obviously, a minimal counterexample to the
Cherlin-Zilber Conjecture is a $K^*$-group.

For a group $G$ of odd type, the crucial parameter is the
\emph{Pr\"{u}fer $2$-rank}, ${\rm pr}(G)$, that is, the number of
copies of\/ $\mathbb{Z}_{2^\infty}$ in the direct sum decomposition $T=
\mathbb{Z}_{2^\infty}\times \cdots \times \mathbb{Z}_{2^\infty}$ of a maximal
divisible $2$-subgroup $T$ of $G$. The present state of affairs is stated in the following theorem,
which summarises a series of works by Alt\i nel, Berkman, Borovik, Burdges, Cherlin, Deloro, Fr\'{e}con, and
Jaligot \cite{ab,berkman,bbrevisited,Bo-SLF,BBBC,BoBu,BBC-ID,BBN-UC,Bu-SF,Bu-GCS,Bu-Bender,Bu09,Bu-Che08,Bu-Che09,BCJ-MCS,BuDe10,Cherlin05,ChJa-Min,Deloro08,Del09b,DeJa10,DJ16,Frecon09,Frecon16} and reduces the classification of groups of odd type to a number of ``small'' configurations.

\begin{fact}
Let $G$ be a simple group of odd type with the property that, for any definable proper subgroup $H$, if $K\lhd H$ is such that $H/K$ is simple, then the latter is an algebraic group.  Then either $G$ is
algebraic or its  Pr\"{u}fer $2$-rank  is at most $2$.
\end{fact}

A  dramatic recent step, resolving a key issue from Cherlin's 1979 paper \cite{Ch-GSR}, was Fr\'econ's 2016 elimination of `bad groups' of Morley rank 3 \cite{Frecon16}. This remarkable result has given new momentum to research around the Cherlin--Zilber Conjecture.

\subsection{Back to permutation groups of finite Morley rank}

In the context of groups of finite Morley rank that are also permutation groups (thus having a definable faithful action), the proof of the following theorem, due to Borovik and Cherlin \cite{borche}, indicates the role of the classification technique: an answer to a basic question about actions of groups of finite Morley rank required  the full strength of the Even Type Theorem (Fact \ref{th:even})  and the full range of techniques developed for the study of groups of odd type, as well as analogues of finite group-theoretic methods (O'Nan--Scott/Aschbacher reductions) from Macpherson and Pillay \cite{Macpherson-Pillay95}.

\begin{fact}{\rm \cite{borche}}\label{Primitive} There exists a function $\rho: \nn \rightarrow
\nn$ with the following property.  If a group $G$ of finite Morley
rank acts on a set\/ $X$ faithfully and definably primitively, then $${\rm rk}(G) \leqslant \rho({\rm
rk}(X)).$$ \end{fact}

Here the action is \emph{definably primitive} if there is no  non-trivial definable equivalence relation on $X$ preserved by $G$.
This result, together with a reasonably well-developed classification of simple groups of finite Morley rank, suggests that definable actions of groups of finite Morley rank allow some form of description and classification. Once it is completed, this classification should take further the theorem by Macpherson and Pillay \cite{Macpherson-Pillay95} that gives a general structural description of primitive groups of finite Morley rank similar to the celebrated O'Nan--Scott Theorem about finite primitive groups \cite{Liebeck-Praeger-Saxl88}.

An analogy with finite group theory may be useful. For finite groups, the classification of finite simple groups (CFSG) --- frequently used through one of its numerous corollaries, the classification of finite 2-transitive permutation groups ---  has had a profound impact, for example on combinatorics and representation theory. In model theory, groups of finite Morley rank naturally appear as definable groups of automorphisms
of structures,  or as binding groups, and one expects that strong structural results on simple groups of finite Morley rank will have a similarly strong impact in model theory. For instance, in finite permutation group theory, CFSG is typically applied through the O'Nan--Scott Theorem and results of Aschbacher on (maximal) subgroups of classical groups; by \cite{Macpherson-Pillay95},  analogous results hold for finite Morley rank.

\subsection{Primitive groups of finite Morley rank}

One of the types of primitive groups of finite Morley rank, called \emph{affine type}, is a semidirect product $H = V \rtimes G$ of two connected groups of finite Morley rank, with $V$ being abelian and $G$ acting faithfully, definably, and irreducibly on $V$ by automorphisms.  (In this setting, if $V$ has no \emph{non-trivial} proper definable $G$-invariant subgroups, we say that $G$ acts \emph{irreducibly} on $V$. We will also use a closely related  term: $V$ is $G$-minimal, or   $G$  acts on $V$  \emph{minimally}, if $V$ contains no \emph{infinite} proper definable $G$-invariant subgroup.) Then it is easy to see that $G$ is a maximal definable subgroup of $H$, and the action of $H$ on the factor space $H/G$ is primitive. All known examples come from rational actions of a reductive algebraic group acting on a unipotent algebraic group, both over the same field $K$, and even in this case it is not known whether $V$ has the structure of a $K$-vector space preserved by $G$. This leads to a major problem:

\begin{problem}
Is it true that if $H = V\rtimes G$ is a primitive group of finite Morley rank of affine type, then $V$ has the structure of a vector space over an algebraically closed field $F$, and the action of $G$ preserves this structure?
\end{problem}

In other words, is it true that if a connected group $G$ of finite Morley rank acts definably and irreducibly on a abelian group $V$, then $V$ is a vector space over some algebraically closed field $F$, and $G$ is a subgroup of $\operatorname{GL}(V)$? The answer is positive in a number of special cases \cite{Borovik-Deloro,chedel,Deloro09}, but, rather surprisingly, in general it remains unknown, even in the category of algebraic groups---with the exception of the characterisation of natural modules for Chevalley groups \cite{Deloro13,Deloro17,Smith89,Timmesfeld2001}, and some modules for ${\rm SL}_2$ close to natural modules \cite{Deloro16A,Deloro16B}.

\begin{problem}
Is it true that if $H = V\rtimes G$ is a an algebraic group over an algebraically closed field $F$, and $V$ is abelian and contains no non-trivial proper $G$-invariant closed subgroups,  then $V$ has the structure of a vector space over $F$, and the action of $G$ preserves this structure?
\end{problem}

The proofs in \cite{borche} show the importance of bounds on generically multiply transitive actions. By definition, a group $G$ acts \emph{generically $n$-transitively} on a set $X$ if the induced action of $G$ on $X^n$ has a generic orbit $A$.
In addition, this action is \emph{generically sharply $n$-transitive} if the stabiliser of a point in  $A$ is trivial. We also say that the action is \emph{generically multiply transitive} if $n >1$.

\begin{fact} {\rm \cite[Corollary 2.2 and Proposition 2.3]{borche}} \begin{itemize}
\item[{\rm (a)}] There is a function $\tau:\mathbb{N}\to \mathbb{N}$ such that for any virtually definably primitive permutation action of a group
$G$ of finite Morley rank that is generically
$t$-transitive on a set $X$ of rank $n$,
$$t\le \tau(n).$$
\item[{\rm (b)}]  For every function $\tau$ as in {\rm (a)}, the function $\rho$ in Fact {\rm \ref{Primitive}} can be chosen so that
$$\rho(n)\le n\tau(n)+{n\choose 2}.$$
\end{itemize} \label{fact:tau}
\end{fact}

It is worth noting that in the case of algebraic groups, the classification of rational generically multiply transitive actions of simple algebraic groups is known only in characteristic $0$ \cite{popov}, and  the group $E_6$ has a generically $4$-transitive action. So we repeat two question asked in \cite{borche}.

\begin{problem}
\bi
 \item[{\rm (a)}] Extend Popov's work \cite{popov} and find all the generically sharply $n$-transitive actions of algebraic groups over algebraically closed fields for $n \geqslant 2$.
 \item[{\rm (b)}] Extend it further to all characteristics and to the finite Morley rank permutation group category in which the groups are Chevalley groups or products of Chevalley groups and tori.   \ei
\end{problem}

\begin{problem} \label{problem-bound-on-generic-transitivity}
Prove that, when $\tau$ is as in Fact~\ref{fact:tau}, then
\[
\tau(n) \leqslant n+2,\
\]
with equality possible only for the natural action of ${\rm PGL}_{n+1}(F)$ on the projective space $\mathbb{P}F^{n}$ over an algebraically closed field $F$.
\end{problem}

\subsection{The result of this paper}

The result proven in this paper deals with one of the crucial configurations in this line of study of primitive actions of groups of finite Morley rank.

\begin{theorem} \label{th:sharply-n-transitive} Let $G$ and $V$ be groups of \fmrd, $V$ a connected  abelian group of Morley rank $n$ without  involutions.  Assume that $G$ acts on $V$ definably, and the action is generically sharply $m$-transitive for $m \geqslant n$.
Then $m = n$, and there is an \acf $F$ such that $V\cong F^n$ and $G\cong\operatorname{GL}(V)$, and the action is the natural action.
\end{theorem}

A group $V \rtimes G$ satisfying the assumptions of Theorem \ref{th:sharply-n-transitive} appears as a point stabiliser in the canonical action of the projective general linear group ${\rm PGL}_{n+1}(F)$ on the projective space $\mathbb{P}F^{n}$; this configuration is unavoidable in any systematic analysis of definable actions of groups of finite Morley rank---in particular, in solving Problem \ref{problem-bound-on-generic-transitivity} and obtaining better bounds in Fact~\ref{Primitive}.

The proof of Theorem~\ref{th:sharply-n-transitive} appears to be deceptively self-contained; indeed it uses only the general theory of groups of  finite Morley rank, with two notable exceptions: the basis of induction, the case $n\leqslant 3$, is a combination of delicate work by Deloro \cite{Deloro09} and a difficult result by Borovik and Deloro \cite{Borovik-Deloro} that uses almost all existing machinery of the classification of groups of odd type, while the final identification of the group $G$ with $\operatorname{GL}(V)$ uses our previous result \cite{bbpseudo}; the proof of the latter required revisiting the older stages of the classification of groups of finite Morley rank \cite{bbrevisited}.

The case when $V$ contains an involution requires an approach different from that of the present paper. It is easy to show that in this case $V$ is an abelian group of exponent $2$. One reason for the need for a special treatment is the non-existence, in characteristic $2$, of the hyperoctahedral group (see Lemma~\ref{lemma:hyperoctahedralgroup}), which is essential in our arguments. Another is the fact that all simple groups of finite Morley rank and of even type are known to be algebraic groups over algebraically closed fields of characteristic $2$ \cite{abc}. Thanks to this fact and other results from \cite{abc}, the problem reduces to the configuration where $G=O(G)*E(G)$ is the central product of two definable subgroups $O(G)$ and $E(G)$; here, $O(G)$ is  the maximal normal definable subgroup of $G$ without involutions, and $E(G)$ is a central product of simple algebraic groups over algebraically closed fields of characteristic $2$. The proof uses heavily the theory of simple algebraic groups and will be published elsewhere.

\subsection{Some immediate future developments}

Of course, it is desirable to remove from Theorem~\ref{th:sharply-n-transitive} assumption of sharpness of the generically $m$-transitive question in question. This is done by methods different from those of the present paper and will be published elsewhere. Our proof involves solving, for the specific case of the hyperoctahedral group $\Sigma_m = Z_2 \wr {\rm Sym}_m$ (see definitions in Section
 \ref{sec:hyperoctahedral}), the following problem.

\begin{problem}{\rm \cite[Problem 8]{borche}}
 Let $\Sigma$ be a finite group. Find  the minimal rank of a connected solvable group of finite Morley
rank that affords a faithful representation of a finite group $\widehat{\Sigma}$ that
covers $\Sigma$, i.e. maps homomorphically onto $\Sigma$.
\end{problem}

\section{Generalities}
\label{sec:generalities}

\subsection{Preliminaries}

Terminology and notation used in this paper follow \cite{abc, bn}. Throughout this paper, all groups are assumed to be of finite Morley rank, all subgroups and actions definable. For a definable set $X$, its Morley rank is denoted $\rk X$.

The structure of nilpotent groups of \fmr is well-known. Recall that a group $G$ is called \emph{divisible} if for every non-zero integer $n$ and every $g\in G$, $x^n=g$ has a solution in $G$.

\begin{fact}{\rm \cite[Proposition I.5.8]{abc}}\label{nilpotentgroups} If $G$ is a connected nilpotent group, then we can express $G=U* R$ as a central product of two of its definable normal subgroups, where $U$ has bounded exponent, $R$ is divisible, and $U\cap R$ is finite.
\end{fact}

Existence of an involutory automorphism may control the structure of the group under certain conditions, as the following fact shows.

\begin{fact} {\rm \cite[Lemma I.10.3]{abc} \cite[Exercises 13, 15, p. 78-79]{bn}} \label{involutoryautomorphisms} Let $G$ be a group and $\varphi$ be a definable automorphism of order $2$ of $G$.
\bi
\item[{\rm (a)}] If $\varphi$ does not fix any non-trivial elements in $G$, then $G$ has no involutions and $\varphi$ is the inversion automorphism on $G$.
\item[{\rm (b)}] If $G$ is connected and $\varphi$ fixes only finitely many elements in $G$, then $\varphi$ is the inversion automorphism on $G$.
\ei
Hence, in both cases, $G$ is abelian.
\end{fact}

A divisible abelian group is called a {\em  decent torus}  if it is the definable hull of its torsion part \cite[Definition I.1.10]{abc}.

\begin{fact}\label{rank1groups} If $G$ is a connected group of Morley rank $1$, then it is abelian. Moreover, one of the following holds: $G$ is an elementary abelian $p$-group,  or torsion free and divisible, or a decent torus.
\end{fact}

\bpr By \cite[Corollary 6.6]{bn}, $G$ is either an elementary abelian $p$-group or a divisible abelian group; further subdivision of the latter follows immediately from $G$ being of Morley rank $1$.
\epr

\begin{fact}\label{automorphismsofptori} A decent torus  does not admit any non-trivial connected definable automorphism groups.

\end{fact}

\begin{proof} Let $T$ be a decent torus and $A$ a connected group acting on $T$. Since $T$ is a divisible abelian group, its torsion part is the direct sum of $p$-tori, that is, finite direct sums of Pr\"{u}fer $p$-groups $\mathbb{Z}_{p^\infty}$ for some primes $p$. Let $P$ be the maximal $p$-torus of $T$, for some prime $p$.  Obviously, $P$ is a characteristic subgroup of $T$ and therefore $A$ leaves $P$ invariant as a set. For each $n\geqslant 1$ set $P_n = \{t\in P: t^{p^n} = 1 \}$. Being characteristic finite subgroups of $P$, all $P_n$'s are centralised by $A$. Since $P=\bigcup P_n$, $A$ also centralises $P$; since the argument works for every prime $p$, $A$ centralises the torsion part of $T$, and hence its definable hull which is equal to $T$ by the definition of a decent torus. Thus, $A=1$. \end{proof}

The following is a corollary of Zilber's Indecomposability Theorem \cite[\S 5.4]{bn}.

\begin{fact} {\rm \cite[Corollary 5.29]{bn}} Let $H$ be a connected subgroup of $G$, and $X$ be a {\rm (}not necessarily definable{\rm )} subset in $G$. Then $[H,X]$ is definable and connected.\label{zilberindthm}
\end{fact}

\begin{fact} {\rm \cite[Corollary 9.9]{bn}} The derived subgroup of a connected solvable group is nilpotent. \label{derived}
\end{fact}

Here are two results from the literature that deal with the case $\rk(V)\leqslant 2$ in Theorem \ref{th:sharply-n-transitive}.

An action on a group is called {\em minimal} if the only improper definable subgroups left invariant under this action are the finite subgroups.

\begin{fact} {\rm (Zilber)} {\rm \cite[Theorem 9.5]{bn}} Let $G$ and $V$ be abelian groups. If $G$ acts on $V$ faithfully and minimally, then there exists an \acf $K$ such that the action $G\curvearrowright V$ is equivalent to the action $B\curvearrowright K^+$ for some subgroup $B$ in $K^*$.           \label{zilberaction}
\end{fact}

\begin{fact} {\rm (Deloro)} {\rm \cite{Deloro09}} \label{deloro} Let\/ $G$ be a connected non-solvable group acting faithfully on a connected abelian group $V$. If $\rk(V)=2$, then there exists an \acf $K$ such that the action $G\curvearrowright V$ is equivalent to $\operatorname{GL}_2(K)\curvearrowright K^2$ or $\operatorname{SL}_2(K)\curvearrowright K^2$.
\end{fact}

In fact, a theorem of Borovik and Deloro deals with our situation for $\rk(V)=3$, however we will not need this result in our proof. A connected non-solvable group of finite Morley rank is called a \emph{bad group} if all its definable connected solvable subgroups are nilpotent.
\begin{fact} {\rm (Borovik-Deloro)} {\rm \cite{Borovik-Deloro}} \label{bordel} Let\/ $G$ be a connected non-solvable group acting faithfully and minimally on an abelian group $V\!$. If\/ $\rk(V)=3$ and $G$ is not a bad group, then there exists an \acf $K$ such that $V=K^3$ and $G$ is isomorphic to either\/ $\operatorname{PSL_2}(K)\times Z(G)$ or\/ $\operatorname{SL_3}(K)* Z(G)$. The action is the adjoint action in the former case, and the natural action in the latter case.
\end{fact}

Here are two results about groups acting on groups.

\begin{fact} {\rm \cite[Proposition I.9.9]{abc}} \label{liftingthecentralisers} Let $H$ be a group of finite Morley rank, $Q \lhd H$ a solvable definable $\pi$-subgroup of bounded exponent and  $t \in H$ a $\pi^\perp$ element.
Then
\[
C_H(t)Q/Q =
C_{H/Q}(t).
\]
\end{fact}

 The abelian group $V$ and its subgroups will be written additively.

\begin{fact} \label{fact:2-action-connected} {\rm \cite[Corollaries I.9.11, I.9.14]{abc}} Let a finite elementary abelian $2$-group $D$ act definably on a connected  abelian group $V$.
Assume that $V$ has no involutions.
Then
\[
V = C_V(D) \oplus [V,D].
\]
In particular, $C_V(D)$ and $[V,D]$ are connected.
\end{fact}

The following three results from our earlier paper \cite{bbpseudo} will be useful in this work as well.

Lemma 7.1 in \cite{bbpseudo}  was stated under stronger assumptions on $V$; however,  the proof used only the fact that $V$ is connected, abelian and has no involutions. So we state Lemma 7.1 in this stronger form:

\begin{fact}\label{fact:2-group} {\rm \cite[Lemma 7.1]{bbpseudo}}
Let $V$ be a connected abelian group and $E$ an elementary abelian $2$-group of order $2^m$ acting definably and faithfully on $V$. Assume $m\geqslant n=\rk(V)$ and $V$ contains no involutions.
Then $m = n$ and $V = V_1 \oplus\cdots\oplus V_n$, where
\bi
\item[{\rm (a)}]
every subgroup $V_i$, $i = 1,\dots,n$, is connected, has Morley rank $1$ and is $E$-invariant.
\ei
Moreover,
\bi
\item[{\rm (b)}] for each $V_i$,  $i = 1,\dots,n$, is a weight space of $E$, that is, there exists a homomorphism $\rho_i:E\to \{\pm 1\}$ such that
     \[
     V_i = \{ v \in V \mid v^e = \rho_i(e)\cdot v \mbox{ for all }   e \in E\}.
     \]
\ei
\end{fact}

\begin{fact} {\rm \cite[Corollary 1.3]{bbpseudo}}\label{bbpseudo:cor1.3} Let $F$ be an \acfd, and $G$ a group acting faithfully on $F^n$ as a group of automorphisms of the additive group of $F^n$. If $\operatorname{GL}_n(F)\leqslant G$ then $G=\operatorname{GL}_n(F)$.
\end{fact}

A $2$-torus (that is, a divisible abelian 2-group) is a product of copies of Pr\" ufer 2-groups $\mathbb C_{2^\infty}$. The number of copies is called the \emph{Pr\" ufer 2-rank} of the 2-torus. If $G$ is a group of finite Morley rank, then the Pr\" ufer 2-rank of $G$ is defined to be the Pr\" ufer 2-rank of a maximal 2-torus in $G$. Since maximal 2-tori are conjugate in $G$ (and have finite Pr\" ufer 2-ranks), the definition is independent of the choice of the maximal 2-torus.

\begin{fact}{\rm \cite[Theorem 1.4]{bbpseudo}}\label{bbpseudo:thm1.4} Let $G$ be a connected group acting on a connected abelian group $V$ faithfully and irreducibly. If the Pr\"ufer 2-rank of $G$ is equal to the Morley rank of $V$, then $V$ is a vector space over an \acf and the action $G\curvearrowright V$ is equivalent to $\operatorname{GL}(V)\curvearrowright V$.
 \end{fact}

 A result of Loveys-Wagner \cite{lw} stated in the below form will be used for the torsion free case of our theorem.

 \begin{fact} {\rm (Loveys-Wagner)} {\rm \cite[Theorem A.20]{bn}}\label{loveyswagner}
 Let $G$ be an infinite group acting on an infinite divisible abelian group $V$. If the action is faithful and $G$-minimal, then there exists an \acf $F$ of characteristic 0 such that $V$ is a vector space over $F$, $G$ is definably isomorphic to a subgroup $H$ of $\operatorname{GL}(V)$, and the action $G\curvearrowright V$ is equivalent to $H\curvearrowright V$.
 \end{fact}

\subsection{Generically multiply transitive actions on sets of Morley degree $\boldsymbol{1}$}\label{subsection:onsets}
A definable subset $Y \subseteq X$ is said to be \emph{generic} in $X$, if $\rk Y = \rk X$.
Assume that a group $G$ is acting on a set $X$ of Morley degree $1$ and that this action is generically sharply $n$-transitive for $n\geqslant 1$, i.e.\ $G$ acts sharply transitively on a generic subset $A$ of $X^n$. Let $\pi_i$ denote the projection from $X^n$ onto the $i$-th component for $i=1,\dots,n$. Then each $\pi_i(A)$ is generic in $X$, and $G$ acts on $\pi_i(A)$ transitively, therefore all $\pi_i(A)$ are equal to the (only) orbit of $G$ generic in $X$. We shall denote $X^*=\pi_i(A)$. Note that $G$ acts transitively and faithfully on $X^*$ and therefore acts faithfully on $X$ by \cite[Lemma 1.6]{borche}.

\begin{proposition} \label{prop:stabiliser-connected}
Under these assumptions, assume that $n>1$. Then
\bi
\item[(a)] Denote the unique generic orbit in $V^m$ by $A$. Then $A$ has Morley degree $1$ and $G$ is connected. Moreover, $\rk G = \rk A =\rk V^m= mn$.
\item[(b)] For all $x\in X^*$, the stabilisers $\stab_G(x)$ are connected.
\ei
\end{proposition}

We shall refer to elements $x\in X^*$ as \emph{generic elements} in $X$.

\begin{proof} By assumption $V$ is connected, hence so is $V^m$.
Observe that $A$ is a generic subset in the set $X^n$ of Morley degree $1$ and therefore also has Morley degree $1$. Since $G$ acts on $A$ sharply transitively, there is a $1$ -- $1$ definable correspondence between  $G$ and $A$ and $G$ also has Morley degree 1. This also shows that $\rk G = \rk A =\rk V^m= mn$ -- this proves (a).

If $n >2$, (b) is an immediate consequence of (a): if $x \in X^*$, its stabiliser $\stab_G(x)$ acts generically sharply $(n-1)$-transitively on $X$. If $n=2$, we apply to the sharp transitive action of $\stab_G(x)$ on $X$ the same argument as in (a).
\end{proof}

\begin{proposition} \label{prop:W-infinite} Let $H$ be a connected group  acting definably on a connected elementary abelian $p$-group $V$ {\rm (}written additively{\rm )}. Assume that $H$ has on $V$ a generic orbit. Assume also that $W= C_V(H)$ is finite and that for a generic $\bar{x} \in \overline{V} = V/W$, the centraliser $C_H(\bar{x})$ is connected.

Then $W = 0$.
\end{proposition}

\begin{proof} By Fact~\ref{zilberindthm}, for a generic $\bar{x} \in \overline{V}$, $[C_H(\bar{x}),x]$ is a connected subgroup of the finite group $W$, hence trivial; if follows that $C_H(\bar{x}) = C_H(x)$.

Observe that $O = x^H$ is generic in $V$.  Fix $w\in W \smallsetminus\{0\}$, then $O \cap (O+w)$ is generic in $V$ because of the connectedness of $V$. Hence there exist generic elements $y$ and $z$ in $V$ with $z =  y+w$ and $z=y^h$ for some $h \in H$. But $\overline{y+w} = \overline{y}$, hence $h \in C_H(\overline{y}) = C_H(y)$, and $z=y^h = y$, and $w=0$. A contradiction.
\end{proof}

\subsection{Some actions of hyperoctahedral groups} \label{sec:hyperoctahedral}

We denote by $Z_2$ the cyclic group of order $2$ and define  $\Sigma_m = Z_2 \wr {\rm Sym}_m$, the wreath product of $Z_2$ and the symmetric group ${\rm Sym}_m$. We denote by $E$ the base group of the wreath product: $E = Z_2\times \cdots\times Z_2$, the direct product of $m$ copies of the group $Z_2$; it is an elementary abelian group of order $2^m$ and could be seen as a vector space of dimension $m$ over the field $\mathbb{F}_2$ with two elements. It is easy to see that $E$ contains two sets of $m$ linearly independent elements which are invariant, setwise, under the action  of ${\rm Sym}_m$; we denote  elements in one of these sets $e_1,\dots, e_m$, then the other set is  $e_1e,\dots, e_me$, where $e = e_1e_2\cdots e_m$.

The group $\Sigma_m$ is called the \emph{hyperoctahedral group}; it is the reflection group of type $BC_m$.

\begin{proposition} \label{prop:hyperoctahedral} Assume that the hyperoctahedral group $\Sigma = \Sigma_m$ acts faithfully and definably on a connected abelian group $V$ {\rm (}written additively{\rm )} of Morley rank $n$ with $m \geqslant n$; we assume, in addition, that $V$ has no involutions.

Then $m=n$ and the following statements are true.
\bi
\item[{\rm (a)}] Set $U_i = [V, e_i]$. Then
\[
V = \bigoplus_{i=1}^n U_i,
\]
\noindent
where $U_i$ are connected  and we may assume without loss of generality that $\rk U_i = 1$ for all $i =1,\dots, n$.

\item[{\rm (b)}] $e_i$ inverts every element in $U_i$:
 $u^{e_i} = -u$  for each $u \in U_i$;
\item[{\rm (c)}]  $e_i$ centralises all $U_j$ for $j \ne i$:
 \[
 C_V(e_i) = \bigoplus_{j\ne i} U_j.
 \]
\item[{\rm (d)}] The group $\Sigma$ transitively permutes subspaces $U_i$, $i = 1,2,\dots,n$.

\item[{\rm (e)}] Every connected $\Sigma$-invariant subgroup of $V$ equals $0$ or $V$.

\item[{\rm (f)}] $C_V(E)=0$ and therefore $C_V(\Sigma)=0$.
\ei
\end{proposition}

\begin{proof} The first
statement and the clauses (a), (b), (c), (d) easily follow from Fact~\ref{fact:2-group}. To prove (e), we  will use Fact~\ref{fact:2-action-connected}; we shall frequently use it in subsequent text without making specific references to it.

Indeed, let $0 \ne W<V$ be a proper definable connected $\Sigma$-invariant subgroup of $V$. If all $U_i$ intersect with $W$ trivially, the group $E$ centralises $W$ and $V = C_V(E) \oplus [V,E]$ by Fact~\ref{fact:2-action-connected}. Hence $E$ acts faithfully on $[V,E]$ and this contradicts with Fact~\ref{fact:2-group}.

Therefore at least one (and hence all, by part (d)) subgroup $U_i$ intersects $W$ non-trivially. Since $e_i$ inverts every element in $U_i$, we have  $W\cap U_i = [W, e_i]$ hence is connected by Fact \ref{fact:2-action-connected}, and, since $U_i$ is a connected group of Morley rank $1$, $W\cap U_i = U_i$, therefore $U_i \leqslant W$ and $W=V$.

To prove (f), it suffices to consider the natural projection of $C_V(E)$ on the direct summands of $\bigoplus_{i=1}^{n} U_i$.
\end{proof}

\section{Proof of Theorem \ref{th:sharply-n-transitive}}

In this section, we work in the setting of Theorem \ref{th:sharply-n-transitive}.

As it was pointed out in the first paragraph of Subsection~\ref{subsection:onsets}, the
generically sharply $m$-transitive action of $G$ on $V$ implies that $G$ acts on $V$ faithfully.

From now on, we use notation of Section \ref{subsection:onsets}; in particular, denote the unique generic orbit in $V^m$ by $A$.

Observe that $G$ is abelian when $m=1$, and hence we are done by a classical result of Zilber (Fact~\ref{zilberaction}) in this case. From now on, we will assume that $m\geqslant 2$.

\begin{lemma} \label{lemma:hyperoctahedralgroup} If\/ $\overline{a} = (v_1,\dots, v_m)$ is an arbitrary $m$-tuple in $A$, then the setwise stabiliser\/ $\Sigma_{\overline{a}}$ in $G$ of the set\/ $\{ \pm v_1, \dots, \pm v_m\}$ is
\[
\Sigma_{\overline{a}} = \langle e_i,s_\sigma\mid 1\leqslant i\leqslant m{\rm ,\ }\sigma\in {\rm Sym}_m\rangle\leqslant  {\rm Sym}_m\ltimes (\mathbb Z_2)^m,
 \]
 where
 \[
 e_i(v_1,\ldots,v_m)=(v_1\ldots,-v_i,\ldots,v_m)
 \]
 and
 \[
 s_\sigma(v_1,\dots,v_m)=(v_{\sigma(1)},\dots,v_{\sigma(m)}).
 \]
In particular, $G$ contains copies of the hyperoctahedral group  as a subgroup, and, moreover, $m =n$.
\end{lemma}

\begin{proof} Clearly the $e_i$'s and elements of ${\rm Sym}_m$ stabilize $\{ \pm v_1, \dots, \pm v_m\}$ setwise. We need to show that they lie in $G$. First we need an observation.

{\em Claim.} If $\rho: V^m\to V^m$ is a definable bijection that commutes with the induced action of $G$ on $V^m$,   then $\rho$ fixes $A$ setwise.

{\em Proof.} If such a $\rho$ exists, then $A\cap\rho A$ is fixed by $G$ setwise. Since $G$ acts transitively on $A$, the intersection $A\cap\rho A$ is either empty or equal to $A$. Since $V$ is connected, the intersection cannot be empty, hence $A\subseteq \rho A$. Now repeat the same thing with $\rho^{-1}$ to get $A=\rho A$. \hfill $\diamond$

Therefore, the following maps fix $A$ setwise:
for every $1\leqslant i\leqslant m$, $r_i:V^m\to V^m$, where
$r_i(v_1,\dots,v_m)=(v_1\dots,-v_i,\dots,v_m)$;
     and for every $\sigma\in{\rm Sym}_m$, $s_\sigma:V^m\to V^m$, where
    $s_\sigma(v_1,\dots,v_m)=(v_{\sigma(1)},\dots,v_{\sigma(m)})$.

Hence, for every $m$-tuple $\overline{a} = (v_1,\dots,v_m)\in A$ and for every $\sigma\in {\rm Sym}_m$, the $m$-tuples   $(\pm v_1,\ldots,\pm v_m)$,  $(v_{\sigma(1)},\dots,v_{\sigma(m)})$ lie in $A$.

Now, by the sharp transitivity of the action of $G$ on $A$ and we obtain the result. To be more precise,
for every $\overline{a} = (v_1,\dots,v_m)\in A$ and every $1\leqslant i\leqslant m$, there exists a unique involution $e_i\in G$ such that $e_i(v_1,\ldots,v_m)=(v_1\ldots,-v_i,\ldots,v_m)$, and also, for every permutation $\sigma\in{\rm Sym}_m$, there exists a unique $s_\sigma\in G$ such that $s_\sigma(v_1,\dots,v_m)=(v_{\sigma(1)},\dots,v_{\sigma(m)})$.

The isomorphism $\Sigma_{\overline{a}} \cong {\rm Sym}_m\ltimes (\mathbb Z_2)^m$ follows from the sharpness of the action, and the equality $m=n$ follows from Proposition \ref{prop:hyperoctahedral}.
\end{proof}

So, from now on $n=m\geqslant 2$. Now we fix one particular $n$-tuple $\overline{\vr{v}} = (\vr{v}_1,\dots, \vr{v}_n) \in A$ and denote $\vr{v}= \vr{v}_1$. The involutions $e_i \in \Sigma_{\overline{\vr{v}}}$ are defined as in Lemma~\ref{lemma:hyperoctahedralgroup}.  We will be using Proposition \ref{prop:hyperoctahedral} which we reproduce here but will be using without specific references to it.

\begin{lemma} \label{lemma:weight-decomposition}
If $U_i = [V, e_i]$ then
\[
V = \bigoplus_{i=1}^n U_i,
\]
where $U_i$ are connected  and $\rk U_i = 1$ for all $i =1,\dots, n$. Observe also that, for each $i$
\bi
\item $\vr{v}_i \in U_i$;
\item $e_i$ inverts every element in $U_i$:  $u^{e_i} = -u$  for each $u \in U_i$;
\item  $e_i$ centralises all $U_j$ for $j \ne i$:
 \[
 C_V(e_i) = \bigoplus_{j\ne i} U_j.
 \]
 \item The group $\Sigma=\Sigma_{\overline{\vr{v}}}$ transitively permutes subspaces $U_i$, $i = 1,2,\dots,n$.
\ei
\end{lemma}

Also, in view of Proposition \ref{prop:hyperoctahedral}, we have the following fundamental observation:

\begin{lemma}\label{lemma:irreducible}
The action of $G$ on $V$ is irreducible, that is, the only definable $G$-invariant subgroups of $V$ are $0$ and $V$.
\end{lemma}

\bpr
Let $0<W<V$ be a $G$-invariant subgroup. Then $W^\circ$ is a connected $G$-invariant subgroup, hence by Proposition \ref{prop:hyperoctahedral}(e), $W^\circ=0$, that is $W$ is finite. If $V$ is torsion free, we immediately conclude that $W=0$. Otherwise, being a connected group, $G$ can act on $W$ only trivially. But then $W \leqslant C_V(\Sigma)=0$ by Proposition \ref{prop:hyperoctahedral}(f).
\epr

\begin{lemma}\label{lemma:char=0}
We can assume without loss of generality that $V$ is an elementary abelian $p$-group for some odd prime $p$.
\end{lemma}

\bpr  By Lemma~\ref{lemma:weight-decomposition}, each $U_i$ is connected and is of Morley rank $1$, hence by Fact~\ref{rank1groups} each $U_i$ is torsion free and divisible, or an elementary abelian $p$-group, or a decent torus. Since $\Sigma$ acts transitively on the set of subgroups $U_i$, they necessarily have the same structure. If each $U_i$ is a decent torus, then $V$ is a decent torus, hence we contradict Fact~\ref{automorphismsofptori}. So we a left with two cases: $V$ is torsion free or $V$ is an elementary abelian $p$-group.

However, we can now easily prove that Theorem~{\rm \ref{th:sharply-n-transitive}} holds if\/ $V$ is a torsion free group.
 Indeed, assume that we are in that situation. By Lemma~\ref{lemma:irreducible}, $G$ acts on $V$ irreducibly. Hence we can apply the Loveys-Wagner Theorem (Fact~\ref{loveyswagner}) and conclude that $V$ is a vector space over an \acf $F$ of characteristic $0$. Notice that the $U_i$'s are non-zero vector subspaces; since they have Morley rank $1$, they are inevitably of dimension $1$ over $F$ and $\rk F =1$. Hence $V$ has dimension $n$ over $F$ and, by Fact~\ref{loveyswagner}, $G$ is definably isomorphic to a subgroup of $\operatorname{GL}_n(F)$, and the actions are equivalent.  Since $\rk(G)=n^2=\rk(\operatorname{GL}_n(F))$ and $\operatorname{GL}_n(F)$ is connected, we can conclude that $G$ is isomorphic to $\operatorname{GL}_n(F)$, and $G\curvearrowright V$ is equivalent to $\operatorname{GL}_n(F)\curvearrowright V$. This proves the torsion free $V$ case.
\epr

From now $V$ is an elementary abelian $p$-group for an odd prime $p$. This remaining case will be proven by induction on $n$, and here we take care of the basis of induction.

\begin{lemma} \label{step:bigger-3}
We can assume that $n\geqslant 3$.
\end{lemma}

\bpr
We have already seen that, without loss of generality, we can assume that $n\neq 1$.
If $n=2$, the theorem follows from a result by Deloro (Fact~\ref{deloro}) on groups acting on abelian groups of Morley rank $2$. We only need to check that $G$ is not solvable. If $G$ is solvable, then the semidirect product $V\ltimes G$ is a connected solvable group. Hence the derived subgroup of $V\ltimes G$ is connected and nilpotent by Fact~\ref{derived} and contains $V$ by Lemma \ref{lemma:irreducible}. The derived subgroup of $\Sigma$ is nilpotent only when $n=2$, since $\Sigma$ is isomorphic to the dihedral group of order $8$; in that specific case the commutator of $\Sigma$ equals $\langle e_1e_2\rangle$ and, being of order $2$ in a nilpotent group which also contains a $p$-group $V$, it centralises $V$, which means that the action of $G$ on $V$ is not faithful, a contradiction.
\epr

We need to introduce some notation. For an arbitrary $x\in V$, denote
\[
A_x = \{ (a_1,\dots, a_n) \in A : a_1 = x\}
\]
and
\[
B_x =\{(a_2,\dots,a_n): (x,a_2,\dots, a_n) \in A\}.
\]

Now note that, with the notation introduced in the first paragraph of Subsection~\ref{subsection:onsets}, $V^*=\pi_i(A)$ for any $i=1,\dots,n$. Moreover, being a generic subset in $V$, $V^*$ has Morley rank $n$. Also note that $A_x$ is non-empty if and only if $x\in V^*$.

\begin{lemma} In this notation, for every $x\in V^*$,
\[
\rk A_x = \rk B_x = n(n-1).
\]
\end{lemma}

\bpr
For $x\in V^*$, the sets $A_x$ form a  partition of $A$ and all have the same Morley rank. 
Hence, for every $x\in V^*$,
 \[
 \rk A_x = \rk A-\rk V^*=n^2-n.
 \]
  Since $A_x$ is in one-to-one correspondence with $B_x$, the result follows.
\epr

We return to analysis of one $n$-tuple $\overline{\vr{v}} = (\vr{v}_1,\dots, \vr{v}_n) \in A$; we denoted $\vr{v}= \vr{v}_1$. We use notation introduced in Lemma \ref{lemma:weight-decomposition}. In addition, we denote  $H= C_G(\vr{v})$ and  $W=C_V(H)$.

\begin{lemma} \label{lemma:W-infinite} \label{lemma:U1-inW} In this notation,

\bi
\item[{\rm (a)}] $H$  is connected and acts sharply transitively on $A_\vr{v}$ and $B_\vr{v}$.
\item[{\rm (b)}] $W$ is infinite.
\item[{\rm (c)}] Every involution $e_i$ normalises $H$ and $W=  C_V(H)$.
\item[{\rm (d)}]$H$ centralises one of the subgroups $U_i$, $i = 1,\dots, n$.
\ei
\end{lemma}

\bpr (a) It is easy to see that $H$ acts on $A_\vr{v}$ and hence on $B_\vr{v}$ sharply transitively. Now the connectedness of $H$ follows from
Proposition~\ref{prop:stabiliser-connected}.

(b) Since $\vr{v}\in W$, $W$ is non-zero. Assume that $W$ is finite and consider $\overline{V} = V/W$ and the action of $H$ on $\overline{V}$. Observe, first of all, that the action of $H$ on $\overline{V}$ is faithful. Indeed, if $h \in C_H(\overline{V})$ then the commutator $[h,V] \leqslant W$ being finite and connected (because $V$ is connected) is trivial. Thus, $h$ centralises $V$ and therefore $h=1$ because $H<G$ acts  on $V$ faithfully. Take the image $\overline{B}_\vr{v}$ in $\overline{V}^{n-1}$; from comparing the ranks we see that $\overline{B}_\vr{v}$ is generic in $\overline{V}^{n-1}$. In view of Proposition~\ref{prop:stabiliser-connected}, we conclude that, for a generic $\bar{x} \in \overline{V}$, $C_H(\bar{x})$ is connected.  We can now apply Proposition \ref{prop:W-infinite} and conclude that $W=0$, a contradiction.

(c) Each $e_i$ normalises $\langle \vr{v}\rangle$, therefore it normalises $H = C_G(\langle \vr{v}\rangle)$  and $W=C_V(H)$.

(d) Let $W_i = [W, e_i] = W \cap U_i$, $i=1,\dots, n$. Since $W$ is infinite by Clause (b), at least one of subgroups $W_k$ in the decomposition $W=W_1 \oplus\cdots\oplus W_n$  is infinite and, being a subgroup of the group $U_k$ of Morley rank $1$ and Morley degree $1$, equals $U_k$. Hence $U_k \leqslant W$.
\epr

 Without loss of generality, we can assume that, in Clause (d) of Lemma~\ref{lemma:U1-inW},  $k=1$ and $U_1 \leqslant W$.

 Now we study $\widetilde{V} = V/U_1$. Let
\[
\alpha: V \longrightarrow \widetilde{V}
\]
be the natural homomorphism. Notice that $\alpha$ preserves the action of the group $H$.
Define $\widetilde{B}_\vr{v}\subseteq \widetilde{V}^{n-1}$ as the image of $B_\vr{v} \subseteq V^{n-1}$
under the componentwise application of the homomorphism $\alpha$.

\begin{lemma} \label{lemma:tilde-B-v-generic} \label{lemma:H-generically-(n-1)-transitive-on-tilde-V}
In this notation,
\bi
\item[{\rm (a)}] $\rk \widetilde{B}_\vr{v} = (n-1)^2$. In particular, $\widetilde{B}_\vr{v}$ is generic in $\widetilde{V}^{n-1}$.

\item[{\rm (b)}] $H$ acts on $\widetilde{V}$ generically $(n-1)$-transitively.
\ei
\end{lemma}

\bpr
(a) Counting ranks of fibers of $\pi$, we get
\begin{eqnarray*}
    \rk \widetilde{B}_\vr{v} &\geqslant& \rk B_\vr{v} - (n-1)\\ &=& n^2-n - (n-1)\\ &=& (n-1)^2.
\end{eqnarray*}
On the other hand, $\widetilde{B}_\vr{v}\subseteq \widetilde{V}^{n-1}$ and the latter has rank $(n-1)^2$. Hence $\rk \widetilde{B}_\vr{v} = (n-1)^2$.

(b) Since $G$ acts transitively on $A$, $H$ acts transitively on $A_\vr{v}$ and $B_\vr{v}$, hence acts transitively on $\widetilde{B}_\vr{v}$ which is generic  in $\widetilde{V}^{n-1}$  by part (a).
\epr

Denote $Q = C_H(\widetilde{V})$,  $K = C_H(e_1)$, and $U= U_2 \oplus \cdots \oplus U_n$.

\begin{lemma}   \label{lemma:e1-fp-free}
$Q$ is an abelian $p$-group of bounded exponent, $e_1$ inverts every element in $Q$ and $H = Q\rtimes K$.

\end{lemma}

\bpr First, we can prove that $Q$ is an abelian $p$-group. Indeed, by construction of $\widetilde{V}$, we have $[V,Q] \leqslant U_1 \leqslant W \leqslant C_V(H) \leqslant C_V(Q)$, hence $[[V,Q],Q] =0$ and by the Three Subgroups Lemma $[[Q,Q], V] = 0$ and hence $[Q,Q]=1$. $V \rtimes Q$ is a nilpotent group, hence by Fact~\ref{nilpotentgroups}, $p$-elements and $p^\perp$-elements in $V \rtimes Q$ commute; it means that $p^\perp$-elements in $Q$ centralise $V$, and therefore equal $1$, that is, $Q$ is a $p$-group.  Using $C_Q(V)=1$ and Fact \ref{nilpotentgroups} one more time gives us that $Q$ is of bounded exponent.

$K \cap Q = 1$. Indeed consider some  $h \in C_H(e_1) \cap Q$ and its action on the $n$-tuple $(\vr{v},\vr{v}_2,\dots,\vr{v}_n)$. Since $h$ centralises $\widetilde{V} = V/U_1$,
\[
\vr{v}_i^h = \vr{v}_i + w_i \quad \mbox{ for some } \quad w_i \in U_1, \quad i=2,\dots, n.
\]
On the other hand, since $h$ centralises $e_1$, it leaves invariant $C_V(e_1)=U$, therefore $\vr{v}_i^h \in  U$ and $w_i \in U_1 \cap U =0$, $i=2,\dots, n$. We should remember also that $\vr{v}^h=\vr{v}$ by definition of $H=C_G(\vr{v})$. Therefore $h$ fixes  $(\vr{v},\vr{v}_2,\dots,\vr{v}_n) \in A$, hence $h=1$.

Now the involution $e_1$ acts on $Q$ without fixed points, therefore the action of $e_1$ on $Q$ is by inversion by Fact~\ref{involutoryautomorphisms}.

The involution $e_1$ centralises $\widetilde{V} = V/U_1$, hence $[H, e_1] \leqslant C_H(\widetilde{V})$, that is, $e_1$ centralises $H/Q$. But $Q$ is a $2^\perp$-group, therefore the centraliser of $e_1$ in $H/Q$ can be lifted to $H$, by Fact~\ref{liftingthecentralisers} applied with $\pi =\{p\}$ and  $t =  e_1$, that is, $H = QC_H(e_1)$. But $K=C_H(e_1)$ intersects with $Q$ trivially, hence $H= Q \rtimes K$.
\epr

\begin{lemma}
$K$ acts sharply generically $(n-1)$-transitively on $U$.
\end{lemma}

\bpr Observe first that, by Lemma \ref{lemma:H-generically-(n-1)-transitive-on-tilde-V}, $K$ acts transitively on $\widetilde{B}_v$ and therefore generically $(n-1)$-transitively on $\widetilde{V}$.

Next we use some basic algebra: the map $\gamma: V \longrightarrow U$ defined by
\[
\gamma(x) = \frac{1}{2}(x+x^{e_1})
\]
is the projection of $V$ onto $U=C_V(e_1)$ with the kernel $U_1$. Also, $\gamma$ preserves the action of $K$.

Finally, since $V = U_1\oplus U$ and $U$ is $K$-invariant, the restriction $\gamma|_U: U \longrightarrow \widetilde{V}$ is a $K$-equivariant isomorphism. We denote by $\beta: \widetilde{V} \longrightarrow U$ the map inverse to $\alpha|_U$, then $\gamma = \beta\circ \alpha$.

The $K$-invariance of the map $\beta$ means that  $K$ acts generically $(n-1)$-transitively on $U$, so we only need to prove the sharpness on the action on $\widetilde{B}_\vr{v}$.  We argue the same way as in proof of Lemma \ref{lemma:e1-fp-free}. Let $L= C_K((\vr{v}_2,\dots, \vr{v}_n))$ and $h \in L$. If $h$ fixes the images $\widetilde{\vr{v}_i}$ in $\widetilde{V}$, then $\vr{v}_i^h = \vr{v}+w_i$ for some $w_i\in U_1$. But $\vr{v}_i \in U$ and $U$ is invariant under action of $h \in K$, hence all $w_i=0$ and $h$ fixes $(\vr{v},\vr{v}_2,\dots, \vr{v}_n)$ hence equals $1$.

\epr

We are now in position to apply the inductive assumption:
\bq
$U$ has a structure of an $(n-1)$-dimensional vector space over an algebraically closed field $F$ and $K$ acts on $U$ as $\operatorname{GL}_{n-1}(F)$ in its natural action on $F^{n-1}$.
\eq

\begin{lemma}[Identifications] \label{lemma:identifications}
We can choose a maximal torus $R$  in $K$ with the following properties:
\bi
\item $R = R_2 \times \cdots \times R_n$, where
\item each $R_i\simeq F^*$ is a torus of Morley rank $1$;
\item $e_i \in R_i$, $i=2,\dots,n$;
\item each $R_i$ acts trivially on $C_V(e_i)$, $i=2,\dots,n$;
\item $[V,R_i]=U_i$ and $C_{U_i}(R_i)=0$, $i=2,\dots, n$;

\ei
\end{lemma}

\bpr Recall that $K \leqslant H$ and that by our choice of notation $U_1 \leqslant W=C_V(H)$, therefore $[K,U_1]=0$.
Generic $(n-1)$-tuples in $U \simeq F^{n-1}$ are linearly independent and are bases of $U$ as a vector space over $F$. Take for $R$ the group of diagonal matrices with respect to the basis $\vr{v}_2,\dots, \vr{v}_n$, then all statements above follow from basic linear algebra and the fact that $[K,U_1]=0$.
\epr

\begin{lemma} $G$ contains a torus $R_1\times R_2 \times \cdots \times R_n$ where each $R_i \simeq F^*$.
\end{lemma}

\bpr Now we pass the information obtained in Lemma~\ref{lemma:identifications} from $H$ to $G$. As an immediate corollary, we have $C_V(R) = U_1$, hence $C_V(K) = C_V(H) = U_1$. Next, $N_H(U)^\circ = K$ by Fact~\ref{bbpseudo:cor1.3}.

By Lemma \ref{lemma:weight-decomposition}, the subspaces $U_1,U_2,\dots, U_n$ are transitively permuted by the group $\Sigma$. Therefore there is a Morley rank 1 torus $R_1$ conjugate to $R_2$, say, by action of $\Sigma$. Then by Lemma \ref{lemma:identifications}, $[U_1, R_1] = U_1$ and $[U,R_1]=0$;  in particular, $R_1$ normalises $U_1$ and $U$. Also, $R_1$ normalises $H = C_G(U_1)$ and therefore $K= N_H(U)^\circ$. Consider the group $L=R_1K$ and notice that $R_1\cap K =1$.

Now note that
\[
T = N_L(U_1) \cap N_L(U_2) \cap \cdots \cap N_L(U_n)= R_1 \times R
\]
is a torus of rank $n$.
\epr

Recall that by Lemma~\ref{lemma:irreducible}, $G$ acts on $V$ irreducibly, hence we can apply Fact~\ref{bbpseudo:thm1.4} and complete the proof. \hfill $\square$

\section*{Acknowledgements}

This paper would have never been written if the authors did not enjoy the warm hospitality offered to them at the Nesin Mathematics Village in \c{S}irince, Izmir Province, Turkey, in summers of 2012 and 2017; our thanks go to Ali Nesin and to all volunteers and staff in the Village.

We thank Gregory Cherlin, Adrien Deloro,  Omar Leon Sanchez, and Dugald Macpherson for fruitful advice, David Pierce for his help with \LaTeX\ and copy-editing, and the anonymous referee, whose comments helped to improve the text.

\section*{References}

\providecommand{\bysame}{\leavevmode\hbox to3em{\hrulefill}\thinspace}
\providecommand{\MR}{\relax\ifhmode\unskip\space\fi MR }

\providecommand{\MRhref}[2]{%
  \href{http://www.ams.org/mathscinet-getitem?mr=#1}{#2}
}
\providecommand{\href}[2]{#2}

\end{document}